\documentclass[oneside,10pt]{article}

\usepackage[b5paper]{geometry}	
\usepackage{amsfonts,amsmath,latexsym,amssymb}
\usepackage{theorem}
\usepackage{mathrsfs,upref}
\usepackage{mathptmx}		
\usepackage{my-mia}	
\usepackage{cite}

\newtheorem{theorem}{Theorem}

\theoremstyle{definition}

\newtheorem{definition}{Definition}

\newtheorem{remark}{Remark}

\allowdisplaybreaks

\begin{document}

\title[On Hermite-Hadamard type inequalities for harmonical $h$-convex interval-valued functions]{
On Hermite-Hadamard type inequalities for harmonical $h$-convex interval-valued functions}

\author{Dafang Zhao, Tianqing An, Guoju Ye and Delfim F. M. Torres}

\address{Dafang Zhao\\
College of Science, Hohai University, Nanjing, Jiangsu 210098, China\\
School of Mathematics and Statistics, Hubei Normal University, Huangshi, Hubei 435002, China\\
\email{dafangzhao@163.com}}

\address{Tianqing An\\
College of Science, Hohai University, Nanjing, Jiangsu 210098, China\\
\email{antq@hhu.edu.cn}}

\address{Guoju Ye\\
College of Science, Hohai University, Nanjing, Jiangsu 210098, China\\
\email{yegj@hhu.edu.cn}}

\address{Delfim F. M. Torres\\
Center for Research and Development in Mathematics and Applications (CIDMA),\\
Department of Mathematics, University of Aveiro, 3810-193 Aveiro, Portugal\\
\email{delfim@ua.pt}}

\CorrespondingAuthor{Delfim F. M. Torres}


\date{Submitted 27.06.2018; Revised 22.10.2019; Accepted 15.11.2019} 




\begin{abstract}
We introduce and investigate the concept of harmonical $h$-convexity
for interval-valued functions. Under this new concept, we prove some 
new Hermite--Hadamard type inequalities for the interval Riemann integral.
\end{abstract}

\maketitle


\section{Introduction}

The following inequality is known in the literature as the Hermite--Hadamard
inequality:
$$
f\Big(\frac{a+b}{2}\Big)
\leq \frac{1}{b-a}\int^{b}_{a}f(x)dx
\leq \frac{f(a)+f(b)}{2},
$$
where $f:I\subseteq\mathbb{R}\rightarrow \mathbb{R}$ is a convex function 
on the interval $I$ and $a, b\in I$ with $a < b$. For various interesting 
extensions and generalizations of this inequality, see
\cite{D17,NNMA16,S10}.
In 2014, \.{I}\c{s}can introduced the concept of harmonical convexity
and established some Hermite--Hadamard type inequalities for this
class of functions \cite{I14}. Some further refinements of such inequalities,
for harmonical convex functions, have been studied in \cite{I162,L17,NNII17}.
In 2015, Noor et al. introduced the class of harmonical $h$-convex functions
and established some Hermite--Hadamard type inequalities \cite{NNAC15}.
For some recent investigations on harmonical $h$-convexity, we refer
the interested readers to \cite{ANMN17,M17,M15}.

On the other hand, interval analysis and interval-valued functions were initially
introduced in numerical analysis by Moore in the celebrated book \cite{M66}.
Because of its wide applications in various fields, interval analysis has emerged
as a very useful research area over the last fifty years: see, e.g.,
\cite{C15,CB,M09} and references therein. Recently, several classical integral 
inequalities have been extended not only to the context of interval-valued
functions by Chalco-Cano et al. \cite{C12,CL15},
Rom\'{a}n-Flores et al. \cite{R18}, Flores-Franuli\v{c} et al. \cite{FC},
Costa and Rom\'{a}n-Flores \cite{CF}, but also to more general set-valued
maps by Klari\v{c}i\'{c} Bakula and Nikodem \cite{K18},
Matkowski and Nikodem \cite{MN}, Mitroi et al. \cite{MNW}, 
and Nikodem et al. \cite{NSS}.

Our research is mainly motivated by the results of \.{I}\c{s}can \cite{I14}
and Noor et al. \cite{NNAC15}. We begin by introducing the notion of
harmonical $h$-convexity for interval-valued functions.
Then we prove some new Hermite--Hadamard type inequalities
for the introduced class of functions. Our inequalities are 
interval-valued counterparts of the results from \cite{I14,NNAC15}.

The paper is organized as follows. After Section~\ref{sec:2} of preliminaries,
in Section~\ref{sec:3} the harmonical $h$-convexity concept
for interval-valued functions is given and new
Hermite--Hadamard type inequalities are proved.
We end with Section~\ref{sec:4} of conclusions and future work.


\section{Preliminaries}
\label{sec:2}

We begin by recalling some basic definitions, notation and properties,
which are used throughout the paper. A real interval $[u]$ is the bounded,
closed subset of $\mathbb{R}$ defined by
$$
[u]=[\underline{u},\overline{u}]
=\{x\in\mathbb{R}|\ \underline{u}\leq x\leq\overline{u}\},
$$
where $\underline{u}, \overline{u}\in \mathbb{R}$
and $\underline{u}\leq\overline{u}$. The numbers
$\underline{u}$ and $\overline{u}$ are called the left
and right endpoints of $[\underline{u},\overline{u}]$,
respectively. When $\underline{u}$ and $\overline{u}$ are equal,
the interval $[u]$ is said to be degenerated.
In this paper, the term interval will mean a nonempty interval.
We call $[u]$ positive if $\underline{u}>0$ or negative
if $\overline{u}<0$. The inclusion ``$\subseteq$'' is defined by
$$
[\underline{u},\overline{u}]\subseteq[\underline{v},\overline{v}]
\Leftrightarrow \underline{v}\leq\underline{u},\overline{u}\leq\overline{v}.
$$
For an arbitrary real number $\lambda$ and $[u]$,
the interval $\lambda [u]$ is given by
\begin{equation*}
\lambda[\underline{u},\overline{u}]=
\begin{cases}
[\lambda\underline{u},\lambda\overline{u}]& \text{if $\lambda>0$},\\
\{0\}& \text{if $\lambda=0$},\\
[\lambda\overline{u},\lambda\underline{u}]& \text{if $\lambda<0$}.
\end{cases}
\end{equation*}
For $[u]=[\underline{u},\overline{u}]$ and $[v]=[\underline{v},\overline{v}]$,
the four arithmetic operators are defined by
$$
[u]+[v]=[\underline{u}+\underline{v},\overline{u}+\overline{v}],
$$
$$
[u]-[v]=[\underline{u}-\overline{v},\overline{u}-\underline{v}],
$$
$$
[u]\cdot[v]=\big[\min\{\underline{u}\underline{v},\underline{u}
\overline{v},\overline{u}\underline{v},\overline{u}\overline{v}\},
\max\{\underline{u}\underline{v},\underline{u}\overline{v},
\overline{u}\underline{v},\overline{u}\overline{v}\}\big],
$$
\begin{equation*}
\begin{split}
[u]/[v]=\big[
&\min\{\underline{u}/\underline{v},\underline{u}/\overline{v},
\overline{u}/\underline{v},\overline{u}/\overline{v}\},\\
&\max\{\underline{u}/\underline{v},\underline{u}/\overline{v},
\overline{u}/\underline{v},\overline{u}/\overline{v}\}\big],
{\rm where}\ \  0\notin[\underline{v},\overline{v}].
\end{split}
\end{equation*}
We denote by $\mathbb{R}_{\mathcal{I}}$ the set of all intervals of
$\mathbb{R}$, and by $\mathbb{R}^{+}_{\mathcal{I}}$ and $\mathbb{R}^{-}_{\mathcal{I}}$
the set of all positive intervals and negative intervals of $\mathbb{R}$, respectively.
The Hausdorff--Pompeiu distance between intervals $[\underline{u},\overline{u}]$
and $[\underline{v},\overline{v}]$ is defined by
$$
d\big([\underline{u},\overline{u}],[\underline{v},\overline{v}]\big)
=\max\Big\{|\underline{u}-\underline{v}|,|\overline{u}-\overline{v}|\Big\}.
$$
It is well known that $(\mathbb{R}_{\mathcal{I}}, d)$
is a complete metric space.

\begin{definition}[See \cite{M79}]
\label{defn2.1}
\rm Let $f:[a,b]\rightarrow \mathbb{R}_{\mathcal{I}}$ be such that
$f(t)=[\underline{f}(t),\overline{f}(t)]$ for each $t\in[a,b]$, and
$\underline{f}$, $\overline{f}$ are Riemann integrable on $[a,b]$.
Then we say that $f$ is Riemann integrable on $[a,b]$ and denote
$$
\int_{a}^{b}f(t)dt=\Bigg[
\int_{a}^{b}\underline{f}(t)dt,
\int_{a}^{b}\overline{f}(t)dt\Bigg].
$$
The collection of all interval-valued functions that are $R$-integrable
on $[a,b]$ will be denoted by $\mathcal{IR}_{([a,b])}$.
\end{definition}

We end this section of preliminaries by recalling some useful known concepts.

\begin{definition}[See \cite{I14}]
\label{defn3.2}
\rm We say that $K_{h}\subset\mathbb{R}\setminus{\{0\}}$ is a harmonical
convex set if
\begin{equation*}
\frac{xy}{t x+(1-t)y}\in K_{h}
\end{equation*}
for all $x,y\in K_{h}$ and $t\in[0,1]$.
\end{definition}

\begin{definition}[See \cite{NNAC15}]
\label{defn3.6}
\rm Let $h:[0,1]\subseteq J\rightarrow \mathbb{R}$
be a non-negative function with $h\not\equiv 0$,
and $K_{h}$ a harmonical convex set. We say that
$f:K_{h}\rightarrow \mathbb{R}$
is a harmonical $h$-convex function if
$$
f\Big(\frac{xy}{t x+(1-t)y}\Big)\leq h(t)f(x)+h(1-t)f(y)
$$
for all $x,y\in K_{h}$ and $t\in[0,1]$.
\end{definition}
Note that if $h(t)=t$, then function $f$ is called a harmonical convex function \cite{I14};
if $h(t)=1$, then $f$ is called a harmonical $P$-convex function \cite{NNAC15};
while, if $h(t)=t^s$, then $f$ is called a harmonical $s$-convex function \cite{NNAC15}.


\section{Main results: new Hermite--Hadamard type inequalities}
\label{sec:3}

In this section, we prove new Hermite--Hadamard type inequalities
for harmonical $h$-convex interval-valued functions.

\begin{definition}
\label{defn3.7}
\rm Let $h:[0,1]\subseteq J\rightarrow \mathbb{R}$
be a non-negative function such that $h\not\equiv 0$,
and $K_{h}$ a harmonical convex set.
We say that $f:K_{h}\rightarrow \mathbb{R}^{+}_{\mathcal{I}}$ is a
harmonical $h$-convex interval-valued function if
\begin{equation}\label{3.2}
h(t) f(x)+h(1-t)f(y) \subseteq f\Big(\frac{xy}{t x+(1-t)y}\Big)
\end{equation}
for all $x,y\in K_{h}$ and $t\in[0,1]$.
If the set inclusion \eqref{3.2} is reversed, then $f$ is said to be
a harmonical $h$-concave interval-valued function.
The set of all harmonical $h$-convex and harmonical $h$-concave
interval-valued functions are denoted by $SX(h,K_{h},\mathbb{R}^{+}_{\mathcal{I}})$
and $SV(h,K_{h},\mathbb{R}^{+}_{\mathcal{I}})$, respectively.
\end{definition}

The next theorem is an interval-valued counterpart 
of \cite[Theorem 3.2]{NNAC15}.

\begin{theorem}
\label{thm3.1}
Let $f:K_{h}\rightarrow \mathbb{R}^{+}_{\mathcal{I}}$ be an interval-valued
function with $a<b$ and $a,b\in K_{h}$, $f\in \mathcal{IR}_{([a,b])}$,
and let $h:[0,1]\rightarrow (0,\infty)$ be a continuous function.
If $f\in SX(h,K_{h},\mathbb{R}^{+}_{\mathcal{I}})$, then
$$
\frac{1}{2h(\frac{1}{2})}f\Big(\frac{2ab}{a+b}\Big)
\supseteq \frac{ab}{b-a}\int^{b}_{a}\frac{f(x)}{x^{2}}dx
\supseteq \big[f(a)+f(b)\big]\int^{1}_{0}h(t)dt.
$$
If $f\in SV(h,K_{h},\mathbb{R}^{+}_{\mathcal{I}})$, then
$$
\frac{1}{2h(\frac{1}{2})}f\Big(\frac{2ab}{a+b}\Big)
\subseteq \frac{ab}{b-a}\int^{b}_{a}\frac{f(x)}{x^{2}}dx
\subseteq \big[f(a)+f(b)\big]\int^{1}_{0}h(t)dt.
$$
\end{theorem}

\begin{proof}
We first assume that $f\in SX(h,K_{h},\mathbb{R}^{+}_{\mathcal{I}})$.
Then one has
$$
h\Big(\frac{1}{2}\Big)f(x)+h\Big(\frac{1}{2}\Big)f(y)
\subseteq f\Big(\frac{xy}{\frac{1}{2} x+\frac{1}{2}y}\Big)
=f\Big(\frac{2xy}{x+y}\Big).
$$
Let
$$
x=\frac{ab}{ta+(1-t)b},\quad
y=\frac{ab}{tb+(1-t)a}.
$$
Then,
\begin{equation}
\label{3.3}
h\Big(\frac{1}{2}\Big)\Bigg[f\Big(\frac{ab}{ta+(1-t)b}\Big)
+f\Big(\frac{ab}{tb+(1-t)a}\Big)\Bigg]\subseteq f\Big(\frac{2ab}{a+b}\Big).
\end{equation}
Integrating both sides of inequality \eqref{3.3} over $[0,1]$, we have
\begin{eqnarray*}
\int^{1}_{0}f\Big(\frac{2ab}{a+b}\Big)dt
&=&\Bigg[\int^{1}_{0}\underline{f}\Big(\frac{2ab}{a+b}\Big)dt,
\int^{1}_{0}\overline{f}\Big(\frac{2ab}{a+b}\Big)dt\Bigg]\\
&=&f\Big(\frac{2ab}{a+b}\Big)\\
&\supseteq& h\left(\frac{1}{2}\right)\int^{1}_{0}\Bigg[
f\Big(\frac{ab}{ta+(1-t)b}\Big)+f\Big(\frac{ab}{tb+(1-t)a}\Big)\Bigg]dt\\
&=&h\left(\frac{1}{2}\right)\Bigg[\int^{1}_{0}\Big(\underline{f}\Big(
\frac{ab}{ta+(1-t)b}\Big)+\underline{f}\Big(\frac{ab}{tb+(1-t)a}\Big)\Big)dt,\\
&&\ \ \ \ \ \ \ \ \ \ \ \ \ \ \ \ \ \ \ \
\int^{1}_{0}\Big(\overline{f}\Big(\frac{ab}{ta+(1-t)b}\Big)
+\overline{f}\Big(\frac{ab}{tb+(1-t)a}\Big)\Big)dt\Bigg]\\
&=&h\left(\frac{1}{2}\right)\Bigg[\frac{2ab}{b-a}\int^{b}_{a}
\frac{\underline{f}(x)}{x^{2}}dx,\frac{2ab}{b-a}\int^{b}_{a}
\frac{\overline{f}(x)}{x^{2}}dx\Bigg]\\
&=&2h\left(\frac{1}{2}\right)\frac{ab}{b-a}
\int^{b}_{a}\frac{f(x)}{x^{2}}dx.
\end{eqnarray*}
This implies that
$$
\frac{1}{2h(\frac{1}{2})}f\Big(\frac{2ab}{a+b}\Big)
\supseteq \frac{ab}{b-a}\int^{b}_{a}\frac{f(x)}{x^{2}}dx.
$$
The proof of the second relation follows by using \eqref{3.2} with $x=a$ and $y=b$
and integrating with respect to $t$ over $[0,1]$, that is,
$$
\frac{ab}{b-a}\int^{b}_{a}\frac{f(x)}{x^{2}}dx
\supseteq \big[f(a)+f(b)\big]\int^{1}_{0}h(t)dt.
$$
The intended result follows. If $f\in SV(h,K_{h},\mathbb{R}^{+}_{\mathcal{I}})$, 
then the proof is similar and is left to the reader.
\hfill \qed
\end{proof}

\begin{remark}
\label{rmk3.1}
\rm If $h(t)=t^{s}$, then Theorem~\ref{thm3.1} gives a result
 for harmonical $s$-functions:
\begin{equation}
\label{3.4}
2^{s-1}f\Big(\frac{2ab}{a+b}\Big)\supseteq \frac{ab}{b-a}
\int^{b}_{a}\frac{f(x)}{x^{2}}dx
\supseteq \frac{1}{s+1}\big[f(a)+f(b)\big].
\end{equation}
If $h(t)=t$, then Theorem~\ref{thm3.1} gives a result
for harmonical convex functions:
\begin{equation}
\label{3.5}
f\Big(\frac{2ab}{a+b}\Big)\supseteq \frac{ab}{b-a}\int^{b}_{a}
\frac{f(x)}{x^{2}}dx\supseteq \frac{f(a)+f(b)}{2}.
\end{equation}
If $h(t)=1$, then Theorem ~\ref{thm3.1} gives a
result for harmonical $P$-functions:
\begin{equation}
\label{3.6}
\frac{1}{2}f\Big(\frac{2ab}{a+b}\Big)\supseteq \frac{ab}{b-a}
\int^{b}_{a}\frac{f(x)}{x^{2}}dx\supseteq f(a)+f(b).
\end{equation}
\end{remark}

\begin{theorem}
\label{thm3.2}
Let $f:K_{h}\rightarrow \mathbb{R}^{+}_{\mathcal{I}}$ be an interval-valued
function with $a<b$ and $a,b\in K_{h}$, $f\in \mathcal{IR}_{([a,b])}$,
and let $h:[0,1]\rightarrow (0,\infty)$ be a continuous function.
If $f\in SX(h,K_{h},\mathbb{R}^{+}_{\mathcal{I}})$, then
\begin{equation*}
\begin{split}
\frac{1}{4\Big[h(\frac{1}{2})\Big]^{2}}f\Big(\frac{2ab}{a+b}\Big)
&\supseteq \Delta_{1}\supseteq \frac{ab}{b-a}\int^{b}_{a}
\frac{f(x)}{x^{2}}dx\supseteq \Delta_{2} \\
&\supseteq
\big[f(a)+f(b)\big]\Big[\frac{1}{2}+h\Big(\frac{1}{2}\Big)\Big]\int^{1}_{0}h(t)dt,
\end{split}
\end{equation*}
where
$$
\Delta_{1}=\frac{1}{4h(\frac{1}{2})}\bigg[
f\Big(\frac{4ab}{a+3b}\Big)+f\Big(\frac{4ab}{3a+b}\Big)\bigg]
$$
and
$$
\Delta_{2}=\bigg[\frac{f(a)+f(b)}{2}
+f\Big(\frac{2ab}{a+b}\Big)\bigg]\int^{1}_{0}h(t)dt.
$$
If $f\in SV(h,K_{h},\mathbb{R}^{+}_{\mathcal{I}})$, 
then the opposite signs of inclusion 
are valid in the above formulas.
\end{theorem}

\begin{proof}
We only give the proof of the first part of Theorem~\ref{thm3.2}.
Since $f\in SX(h,K_{h},\mathbb{R}^{+}_{\mathcal{I}})$, we have
$$
f\Big(\frac{2xy}{x+y}\Big)\supseteq h\Big(\frac{1}{2}\Big)\Big[f(x)+f(y)\Big]
$$
for all $x,y\in K_{h}$ and $t=\frac{1}{2}$. Choosing
$$
x=\frac{a\frac{2ab}{a+b}}{ta+(1-t)\frac{2ab}{a+b}},
\quad y=\frac{a\frac{2ab}{a+b}}{t\frac{2ab}{a+b}+(1-t)a},
$$
we get
$$
f\Big(\frac{4ab}{a+3b}\Big)\supseteq h\Big(\frac{1}{2}\Big)
\Bigg[f\Big(\frac{a\frac{2ab}{a+b}}{ta+(1-t)
\frac{2ab}{a+b}}\Big)+f\Big(
\frac{a\frac{2ab}{a+b}}{t\frac{2ab}{a+b}+(1-t)a}\Big)\Bigg].
$$
Integrating both sides of the above inequality over $[0,1]$, we have
\begin{eqnarray*}
f\Big(\frac{4ab}{a+3b}\Big)
&\supseteq& h\Big(\frac{1}{2}\Big)\Bigg[\int_{0}^{1}\underline{f}\Big(
\frac{a\frac{2ab}{a+b}}{ta+(1-t)\frac{2ab}{a+b}}\Big)+\underline{f}\Big(
\frac{a\frac{2ab}{a+b}}{t\frac{2ab}{a+b}+(1-t)a}\Big)dt,\\
&&\ \ \ \ \ \ \ \ \ \ \ \ \ \ \ \ \ \ \ \
\int_{0}^{1}\overline{f}\Big(\frac{a\frac{2ab}{a+b}}{ta+(1-t)
\frac{2ab}{a+b}}\Big)+\overline{f}\Big(\frac{a\frac{2ab}{a+b}}{t
\frac{2ab}{a+b}+(1-t)a}\Big)dt\Bigg]\\
&=&h\Big(\frac{1}{2}\Big)\frac{4ab}{b-a}\Bigg[\int_{a}^{\frac{2ab}{a+b}}
\frac{\underline{f}(x)}{x^{2}}dx,\int_{a}^{\frac{2ab}{a+b}}
\frac{\overline{f}(x)}{x^{2}}dx\Bigg]\\
&=&h\Big(\frac{1}{2}\Big)\frac{4ab}{b-a}
\int_{a}^{\frac{2ab}{a+b}}\frac{f(x)}{x^{2}}dx.
\end{eqnarray*}
Similarly, we have
$$
f\Big(\frac{4ab}{3a+b}\Big)\supseteq h\Big(\frac{1}{2}\Big)
\frac{4ab}{b-a}\int_{\frac{2ab}{a+b}}^{b}\frac{f(x)}{x^{2}}dx.
$$
Consequently, we get
\begin{equation*}
\begin{split}
f\Big(\frac{2ab}{a+b}\Big)
&\supseteq h\Big(\frac{1}{2}\Big)\Bigg[f\Big(
\frac{4ab}{a+3b}\Big)+f\Big(\frac{4ab}{3a+b}\Big)\Bigg]
=4\Bigg[h\Big(\frac{1}{2}\Big)\Bigg]^{2}\Delta_{1}\\
&\supseteq4\Bigg[h\Big(\frac{1}{2}\Big)\Bigg]^{2}
\frac{ab}{b-a}\int_{a}^{b}\frac{f(x)}{x^{2}}dx.
\end{split}
\end{equation*}
Thanks to Theorem~\ref{thm3.1},
\begin{equation*}
\begin{split}
\frac{1}{4\Big[h(\frac{1}{2})\Big]^{2}}f\Big(\frac{2ab}{a+b}\Big)
&\supseteq\Delta_{1}
\supseteq \frac{ab}{b-a}\int^{b}_{a}\frac{f(x)}{x^{2}}dx\\
&=\frac{1}{2}\Bigg[\frac{2ab}{b-a}\int^{\frac{2ab}{a+b}}_{a}\frac{f(x)}{x^{2}}dx
+\frac{2ab}{b-a}\int^{b}_{\frac{2ab}{a+b}}\frac{f(x)}{x^{2}}dx\Bigg]\\
&\supseteq \frac{1}{2}\bigg[f(a)+f(b)+2f\Big(\frac{2ab}{a+b}\Big)\bigg]
\int^{1}_{0}h(t)dt
=\Delta_{2}\\
&\supseteq \bigg[\frac{f(a)+f(b)}{2}+h\Big(\frac{1}{2}\Big)f(a)
+h\Big(\frac{1}{2}\Big)f(b)\bigg]\int^{1}_{0}h(t)dt\\
&=\big[f(a)+f(b)\big]\Big[\frac{1}{2}
+h\Big(\frac{1}{2}\Big)\Big]\int^{1}_{0}h(t)dt
\end{split}
\end{equation*}
and the result follows.
\hfill \qed
\end{proof}

\begin{remark}
\label{rmk3.2}
Like in Remark~\ref{rmk3.1}, from Theorem~\ref{thm3.2} we obtain particular 
results for harmonical convex, harmonical $P$-convex, 
and harmonical $s$-convex functions.
\end{remark}

\begin{theorem}
\label{thm3.3}
Let $f, g:K_{h}\rightarrow \mathbb{R}^{+}_{\mathcal{I}}$ be interval-valued
functions with $a<b$ and $a,b\in K_{h}$, $fg\in \mathcal{IR}_{([a,b])}$,
and $h_{1}, h_{2}:[0,1]\rightarrow (0,\infty)$ be continuous functions.
If $f\in SX(h_{1},K_{h},\mathbb{R}^{+}_{\mathcal{I}})$,
$g\in SX(h_{2},K_{h},\mathbb{R}^{+}_{\mathcal{I}})$, then
$$
\frac{ab}{b-a}\int^{b}_{a}\frac{f(x)g(x)}{x^{2}}dx
\supseteq M(a,b)\int^{1}_{0}h_{1}(t)h_{2}(t)dt
+N(a,b)\int^{1}_{0}h_{1}(t)h_{2}(1-t)dt,
$$
where
$$
M(a,b)=f(a)g(a)+f(b)g(b),
\quad
N(a,b)=f(a)g(b)+f(b)g(a).
$$
If $f\in SV(h_{1},K_{h},\mathbb{R}^{+}_{\mathcal{I}})$,
$g\in SV(h_{2},K_{h},\mathbb{R}^{+}_{\mathcal{I}})$, then
$$
\frac{ab}{b-a}\int^{b}_{a}\frac{f(x)g(x)}{x^{2}}dx
\subseteq M(a,b)\int^{1}_{0}h_{1}(t)h_{2}(t)dt
+N(a,b)\int^{1}_{0}h_{1}(t)h_{2}(1-t)dt.
$$
\end{theorem}

\begin{proof}
By hypothesis, one has
$$
f\Big(\frac{ab}{ta+(1-t)b}\Big)\supseteq h_{1}(t)f(a)+h_{1}(1-t)f(b),
$$
$$
g\Big(\frac{ab}{ta+(1-t)b}\Big)\supseteq h_{2}(t)g(a)+h_{2}(1-t)g(b).
$$
Then,
\begin{equation*}
\begin{split}
&f\Big(\frac{ab}{ta+(1-t)b}\Big)g\Big(\frac{ab}{ta+(1-t)b}\Big)\\
&\supseteq h_{1}(t)h_{2}(t)f(a)g(a)+h_{1}(t)h_{2}(1-t)f(a)g(b)\\
&\ \ \ +h_{1}(1-t)h_{2}(t)f(b)g(a)+h_{1}(1-t)h_{2}(1-t)f(b)g(b).
\end{split}
\end{equation*}
Integrating both sides of the above inequality over $[0,1]$, we have
\begin{eqnarray*}
&&\int^{1}_{0}f\Big(\frac{ab}{ta+(1-t)b}\Big)g\Big(\frac{ab}{ta+(1-t)b}\Big)dt\\
&=&\int^{1}_{0}\Bigg[\underline{f}\Big(\frac{ab}{ta+(1-t)b}\Big)\underline{g}
\Big(\frac{ab}{ta+(1-t)b}\Big),\overline{f}\Big(\frac{ab}{ta+(1-t)b}\Big)
\overline{g}\Big(\frac{ab}{ta+(1-t)b}\Big)\Bigg]dt\\
&=&\Bigg[\int^{1}_{0}\underline{f}\Big(\frac{ab}{ta+(1-t)b}\Big)
\underline{g}\Big(\frac{ab}{ta+(1-t)b}\Big)dt,\int^{1}_{0}\overline{f}
\Big(\frac{ab}{ta+(1-t)b}\Big)\overline{g}\Big(\frac{ab}{ta+(1-t)b}\Big)dt\Bigg]\\
&=&\Bigg[\frac{ab}{b-a}\int^{b}_{a}\frac{\underline{f}(x)\underline{g}(x)}{x^{2}}dx,
\frac{ab}{b-a}\int^{b}_{a}\frac{\overline{f}(x)\overline{g}(x)}{x^{2}}dx\Bigg]\\
&=&\frac{ab}{b-a}\int^{b}_{a}\frac{f(x)g(x)}{x^{2}}dx\\
&\supseteq& f(a)g(a)\int^{1}_{0}h_{1}(t)h_{2}(t)dt+f(a)g(b)\int^{1}_{0}h_{1}(t)h_{2}(1-t)dt\\
&&\ \ \ +f(b)g(a)\int^{1}_{0}h_{1}(1-t)h_{2}(t)dt+f(b)g(b)\int^{1}_{0}h_{1}(1-t)h_{2}(1-t)dt\\
&=&M(a,b)\int^{1}_{0}h_{1}(t)h_{2}(t)dt +N(a,b)\int^{1}_{0}h_{1}(t)h_{2}(1-t)dt.
\end{eqnarray*}
This concludes the proof.
\hfill \qed
\end{proof}

\begin{remark}
Similarly as before, from Theorem~\ref{thm3.3} we obtain particular 
results for harmonical convex, harmonical $P$-convex, 
and harmonical $s$-convex functions.
\end{remark}

\begin{theorem}
\label{thm3.4}
Let $f, g:K_{h}\rightarrow \mathbb{R}^{+}_{\mathcal{I}}$ be interval-valued
functions with $a<b$, where $a,b\in K_{h}$ and $f g\in \mathcal{IR}_{([a,b])}$, 
and $h_{1}, h_{2}:[0,1]\rightarrow (0,\infty)$ be continuous functions. If 
$f\in SX(h_{1},K_{h},\mathbb{R}^{+}_{\mathcal{I}})$
and $g\in SX(h_{2},K_{h},\mathbb{R}^{+}_{\mathcal{I}})$, then
\begin{multline}
\label{eq:rel:thm:3.4}
\frac{1}{2h_{1}(\frac{1}{2})h_{2}(\frac{1}{2})}f\left(\frac{2ab}{a+b}\right)
g\left(\frac{2ab}{a+b}\right)
\supseteq\frac{ab}{b-a}\int^{b}_{a}\frac{f(x)g(x)}{x^{2}}dx\\
+ M(a,b)\int^{1}_{0}h_{1}(t)h_{2}(1-t)dt
+N(a,b)\int^{1}_{0}h_{1}(t)h_{2}(t)dt.
\end{multline}
If $f\in SX(h_{1},K_{h},\mathbb{R}^{+}_{\mathcal{I}})$ and
$g\in SX(h_{2},K_{h},\mathbb{R}^{+}_{\mathcal{I}})$, then
previous formula \eqref{eq:rel:thm:3.4} holds with 
the opposite sign of inclusion.
\end{theorem}

\begin{proof}
Let $\xi = \frac{2ab}{a+b}$. By hypothesis, one has
$$
f(\xi)\supseteq h_{1}\Big(\frac{1}{2}\Big)
f\Big(\frac{ab}{ta+(1-t)b}\Big)
+h_{1}\Big(\frac{1}{2}\Big)f\Big(\frac{ab}{tb+(1-t)a}\Big),
$$
$$
g(\xi)\supseteq
h_{2}\Big(\frac{1}{2}\Big)g\Big(\frac{ab}{ta+(1-t)b}\Big)
+h_{2}\Big(\frac{1}{2}\Big)g\Big(\frac{ab}{tb+(1-t)a}\Big).
$$
Then,
\begin{small}
\begin{eqnarray*}
&&f(\xi)g(\xi)\\
&\supseteq&h_{1}\Big(\frac{1}{2}\Big)h_{2}\Big(\frac{1}{2}\Big)\Bigg[f\Big(\frac{ab}{ta
+(1-t)b}\Big)g\Big(\frac{ab}{ta+(1-t)b}\Big)+f\Big(\frac{ab}{tb+(1-t)a}\Big)g
\Big(\frac{ab}{tb+(1-t)a}\Big)\Bigg]\\
&&\ \ +h_{1}\Big(\frac{1}{2}\Big)h_{2}\Big(\frac{1}{2}\Big)\Bigg[f\Big(\frac{ab}{ta+(1-t)b}
\Big)g\Big(\frac{ab}{tb+(1-t)a}\Big)+f\Big(\frac{ab}{tb+(1-t)a}\Big)g\Big(\frac{ab}{ta+(1-t)b}\Big)\Bigg]\\
&\supseteq& h_{1}\Big(\frac{1}{2}\Big)h_{2}\Big(\frac{1}{2}\Big)\Bigg[f\Big(\frac{ab}{ta+(1-t)b}\Big)
g\Big(\frac{ab}{ta+(1-t)b}\Big)+f\Big(\frac{ab}{tb+(1-t)a}\Big)g\Big(\frac{ab}{tb+(1-t)a}\Big)\Bigg]\\
&&\ \ +h_{1}\Big(\frac{1}{2}\Big)h_{2}\Big(\frac{1}{2}\Big)\Big[\big(h_{1}(t)f(a)
+h_{1}(1-t)f(b)\big)\big(h_{2}(1-t)g(a)+h_{2}(t)g(b)\big)\\
&&\ \ \ +\Big(h_{1}(1-t)f(a)+h_{1}(t)f(b)\Big)\Big(h_{2}(t)g(a)
+h_{2}(1-t)g(b)\Big)\Big]\\
&=&h_{1}\Big(\frac{1}{2}\Big)h_{2}\Big(\frac{1}{2}\Big)\Bigg[f\Big(
\frac{ab}{ta+(1-t)b}\Big)g\Big(\frac{ab}{ta+(1-t)b}\Big)
+f\Big(\frac{ab}{tb+(1-t)a}\Big)g\Big(\frac{ab}{tb+(1-t)a}\Big)\Bigg]\\
&&\ \ +h_{1}\Big(\frac{1}{2}\Big)h_{2}\Big(\frac{1}{2}\Big)\Big[
\Big(h_{1}(t)h_{2}(1-t)+h_{1}(1-t)h_{2}(t)\Big)M(a,b)\\
&&\ \ \ \ \ \ \ +\Big(h_{1}(t)h_{2}(t)+h_{1}(1-t)h_{2}(1-t)\Big)N(a,b)\Big].
\end{eqnarray*}
Integrating over $[0,1]$, we have
\begin{multline*}
\frac{1}{2h_{1}(\frac{1}{2})h_{2}(\frac{1}{2})}
f(\xi)g(\xi)
\supseteq\frac{ab}{b-a}\int^{b}_{a}\frac{f(x)g(x)}{x^{2}}dx\\
+ M(a,b)\int^{1}_{0}h_{1}(t)h_{2}(1-t)dt
+N(a,b)\int^{1}_{0}h_{1}(t)h_{2}(t)dt.
\end{multline*}
\end{small}
This concludes the proof.
\hfill \qed
\end{proof}

\begin{remark}
From Theorem~\ref{thm3.4}, we obtain particular results for
harmonical convex, harmonical $P$-convex, 
and harmonical $s$-convex functions.
\end{remark}


\section{Conclusions}
\label{sec:4}

We introduced the new concept of harmonical $h$-convexity
for interval-valued functions. Some interesting Hermite--Hadamard
type inequalities for harmonical $h$-convex interval-valued
functions have then been proved. Our results give interval-valued
counterparts of the inequalities presented by \.{I}\c{s}can
and Noor et al., respectively in \cite{I14} and \cite{NNAC15}.

Further developments are possible. As a future research direction,
we intend to investigate Hermite--Hadamard type inequalities for
harmonical $h$-convex interval-valued functions on arbitrary
time scales.


\section*{Acknowledgements}

This research is supported by the Fundamental Research Funds
for the Central Universities (2017B19714, 2017B07414 and 2019B44914),
and in part by Special Soft Science Research Projects of Technological
Innovation in Hubei Province (2019ADC146), Key Projects of Educational
Commission of Hubei Province of China (D20192501) and the Natural Science
Foundation of Jiangsu Province (BK20180500).
Torres was supported by FCT and CIDMA, project UID/MAT/04106/2019.

The authors would like to thank an anonymous reviewer for his/her critical remarks 
and precious suggestions, which helped them to improve the quality and clarity 
of the manuscript.



\end{document}